\documentclass[12pt]{amsart}
\usepackage{amsthm, amssymb, amsmath}
\usepackage{amsfonts, amscd, color}
\usepackage{epsfig,multicol}
\usepackage{tikz-cd}
\usepackage[colorlinks=true, linkcolor=blue, citecolor=blue, urlcolor=blue]{hyperref}

\theoremstyle{plain} \numberwithin{equation}{section}
\newtheorem{theo}{Theorem}[section]

\theoremstyle{definition}
\newtheorem*{defi}{Definition}

\newtheorem*{thm}{Theorem}

\def\C{\mathbb C}
\def\CS{\mathbb C^*}

\def\CZN{\C^n\times\Z^n}
\def\N{\mathbb N}
\def\R{\mathbb R}

\def\T{\mathbb T}
\def\Z{\mathbb Z}
\def\b{b}
\def\v{v}
\def\i{\sqrt{-1}}

\def\ME{\mathcal E}

\def\MO{\mathcal O}
\def\MR{\mathcal R}

\def\OPHI{\overline{\phi}}

\def\De{\Delta}

\def\la{\lambda}

\def\RP{\R_{>0}}
\def\Si{\Sigma}

\def\bp{\beta_{I}^{\perp}}

\def\1{\mathbf 1}
\def\0{\mathbf 0}

\DeclareMathOperator{\Ker}{Ker}
\DeclareMathOperator{\Hom}{Hom}

\DeclareMathOperator{\diag}{diag}
\DeclareMathOperator{\GL}{GL}
\DeclareMathOperator{\TKVB}{TKVB}
\DeclareMathOperator{\TKVS}{TKVS}

\DeclareMathOperator{\SKVS}{SKVS}
\DeclareMathOperator{\Spec}{Spec}

\begin{document}

\title{Equivariant vector bundles over topological toric manifolds }
\author[Y.Cui, A. Gholampour]{Yong Cui and Amin Gholampour} 
\address{Department of Mathematics, University of Maryland, Maryland, MD, USA}
\email{cuiyong@umd.edu, amingh@umd.edu}

\begin{abstract}
We prove that every topological/smooth $\T=(\C^{*})^{n}$-equivariant vector bundle over a topological toric manifold of dimension $2n$ is a topological/smooth Klyachko vector bundle in the sense of \cite{Cui25}.
\end{abstract}

\maketitle

\tableofcontents
\thispagestyle{empty}

\section{Introduction}
The equivariant algebraic vector bundles over nonsingular complete toric varieties were completely classified in \cite{Kly90} by Klyachko. In \cite{Cui25}, the first author generalized Klyachko's result to a special class of equivariant topological/smooth vector bundles over topological toric manifolds called topological/smooth Klyachko vector bundles. 

A \emph{topological toric manifold} is a closed smooth real manifold of dimension $2n$ with an effective smooth action of $\T=(\CS)^{n}$ having an open dense orbit, which can be covered by finitely many invariant open subsets each equivariantly diffeomorphic to a direct sum of complex one-dimensional smooth representation spaces of $\T$ (see \cite{{IFM12}} and Section \ref{background} for more details). The topological toric manifolds are a differential geometric generalization of nonsingular complete toric varieties over $\C$, and their geometry is completely controlled by their topological fans. To each top dimensional cone $I$ in the fan, there corresponds an affine chart $U_{I}$. This is similar to the correspondence between cones $\sigma$ and affine charts $U_{\sigma}$ in toric geometry. However, the transition functions between $U_I$'s are given by the smooth characters of $\T$, which are in general far from being algebraic. 

A \emph{topological/smooth Klyachko vector bundle} is simply a continuous/smooth $\T$-equivariant vector bundles over a topological toric manifold $X$ that is $\T$-equivariantly trivial over all charts $U_{I}$. In \cite{Cui25}, we asked the following question: \\
$ $\\
\noindent \textbf{Question: }
 Is every  $\T$-equivariant vector bundle over a topological toric manifold a Klyachko vector bundle?   \\
$ $\\
In this paper, we will provide an affirmative answer to this question. In the case of toric vector bundles over nonsingular toric varieties, this was proven in Proposition 2.1.1. of~\cite{Kly90}. Let $\sigma$ be a cone in the fan of a complete toric variety $X$, $U_\sigma=\Spec\C[\sigma^{\vee}]$ be the affine toric variety corresponding to $\sigma$, and $x_{\sigma}$ be an arbitrary point in the orbit $\MO_{\sigma}$ corresponding to $\sigma$ fixed by the subtorus $\T_\sigma<\T$. For a toric vector bundle  $\ME$ over $X$ with the fiber  $\ME(x_{\sigma})$ over  over $x_{\sigma}$, let 
\begin{align*}
    p:\Gamma(U_\sigma,\ME)&\longrightarrow \ME(x_{\sigma}) \qquad 
    s \longmapsto s(x_{\sigma})
\end{align*}
be the evaluation map at $x_\sigma$. A key fact used in Klyachko's proof of equivariant triviality of $\ME|_{U_\sigma}$ is that for any $\T_\sigma$-eigenvector $v\in \ME(x_{\sigma})$, there is an eigensection $s\in \Gamma(U_\sigma,\ME)$, such that $p(s)=v$. This follows from the following two theorems and the surjectivity of $p$, which is clear as $p$ is the canonical projection onto a quotient of the $\C[\sigma^\vee]$-module $\Gamma(U_\sigma,\ME(x_\sigma))$.
\begin{thm} [Lemma* on page 25 of \cite{MFK94}]
   Any rational representation of a linear algebraic group is a union of finite-dimensional invariant subspaces. 
\end{thm}

\begin{thm} [{\cite[Theorem 3.2.3]{Spr09}}]
    Any finite dimensional rational representation of a diagonalizable linear algebraic group $G$ is a direct sum of one dimensional such representations.
\end{thm}

%Denote $\Gamma(X,\ME)$ by $M$ and the maximal ideal of $\C[\sigma^{\vee}]$ corresponding to $x_{\sigma}$ by $m_{\sigma}$, then the map 
%\[p: M\to M/m_{\sigma}M\]
%is surjective. 

In the case of an equivariant vector bundle $\ME$ over an affine chart $U_{I}$ of a topological toric manifold $X$, $\Gamma(U_{I},\ME)$ is a \emph{continuous/smooth representation space} of $\T$ to which Theorems above do not apply. 
Moreover, the surjectivity of the evaluation map $p:\Gamma(U_{I},\ME)\to \ME(x_{I})$ is not clear to us in general.  We will bypass these difficulties by considering the fixed point $x_{I}$ of each affine chart $U_{I}$ and constructing a $G=(S^{1})^{n}$-eigenframe on a small neighborhood of $x_{I}$, and extending this frame to a $\T$-eigenframe defined on $U_{I}$. 

 The paper is organized as follows: In §1, we review the construction of topological toric manifolds. In §2, we state the orbit-cone correspondence for topological toric manifolds and define topological/smooth Klyachko vector bundles. In §3, we prove our main result that every topological/smooth $\T$-equivariant vector bundle is a topological/smooth Klyachko vector bundle.
\bigskip

\noindent
{\bf Acknowledgment.}  
The first author would like to thank Professor Sam Payne for answering many questions about toric vector bundles during the preparation of this paper.
\bigskip

\section{Background on topological toric manifolds} \label{background}
Topological toric manifolds are generalizations of smooth toric varieties. Like toric varieties, they are determined by topological toric fans. Here we review the construction in \cite{IFM12}.

 Let $S$ be the ring of $2\times 2$ matrices of the form $\begin{bmatrix}
     b &0\\
     c &v
 \end{bmatrix}$ where $b,c\in \R,v\in \Z$.
Denote the set $\C\times \Z$  by $\MR$. We define a ring structure on $\MR$ via the bijection $\MR\to S$ sending $\mu=(b+\i c,v)$ to $\begin{bmatrix}
     b &0\\
     c &v
 \end{bmatrix}$, i.e. 
 \begin{align*}
  \mu_1+\mu_2&:=((b_1+b_2)+\sqrt{-1}(c_1+c_2),v_1+v_2)\\  
  \mu_1\mu_2&:=(b_1b_2+\sqrt{-1}(b_1c_2+c_1v_2),v_1v_2).
 \end{align*}
 for $\mu_i=(b_i+\sqrt{-1}c_i,v_i),i=1,2.$\\
 For $g\in \CS, \mu=(b+\i c,v)\in \MR$, define 
 \[
 g^\mu:=|g|^{b+\i c}(\frac{g}{|g|})^v,
 \]
then $(g^{\mu_1})^{\mu_2}=g^{\mu_2\mu_1}$.\\
The identity element $\1=(1,1)$ of $\MR$ corresponds to the identity matrix in $S$.
The ring $\MR^{n}$ plays the roles of the (co)character group  for topological toric manifolds.
 
Let $G$ be an abelian topological group. A continuous (resp.\hspace{0.1cm}smooth) character of $G$ is a continuous (resp.\hspace{0.1cm}smooth) homomorphism $G\to \CS$. A continuous (resp.\hspace{0.1cm}smooth) cocharacter of $G$ is a continuous (resp.\hspace{0.1cm}smooth) homomorphism $\CS\to G$. \\
 For $\alpha=(\alpha^1,\dots,\alpha^n)\in \MR^n$ and $\beta=(\beta^1,\dots,\beta^n)\in \MR^n$, define a smooth character $\chi^\alpha\in \Hom((\CS)^n,\CS)$ and a smooth cocharacter $\la_\beta\in \Hom(\CS,(\CS)^n)$ via
 \[
 \chi^\alpha(g_1,\dots,g_n):=\prod_{k=1}^n g_k^{\alpha^k},  \qquad\la_\beta(g):=(g^{\beta^1},\dots,g^{\beta^n}).
 \]

 All continuous characters (resp. cocharacters) of $\T$ are of the form $\chi^{\alpha}$ (resp. $\la_{\beta}$). In particular, every continuous character of $\T$ is smooth.    
There is a bracket operation on $\MR$ defined by 
 \[
 \langle \alpha,\beta\rangle:=\sum_{k=1}^n \alpha^k\beta^k\in \MR.
 \]
 Notice that in general $\langle \alpha,\beta\rangle\neq \langle \beta,\alpha\rangle$ because the matrix product is not commutative.
 
 For $\{\alpha_i\}_{i=1}^{n}$ and $\{\beta_i\}_{i=1}^{n}$ with $\alpha_i,\beta_i\in \MR^n$, we define group endomorphisms $\oplus_{i=1}^{n}\chi^{\alpha_{i}}$ and $\prod_{i=1}^{n}\la_{\beta_{i}}$ of $(\CS)^{n}$ by 
\begin{equation} 
\begin{split}
&(\bigoplus_{i=1}^n\chi^{\alpha_i})(g_1,\dots,g_n):=(\chi^{\alpha_1}(g_1,\dots,g_n),\dots, \chi^{\alpha_n}(g_1,\dots,g_n))\\
&(\prod_{i=1}^n\la_{\beta_i})(g_1,\dots,g_n):=\prod_{i=1}^n\la_{\beta_i}(g_i).
\end{split}
\end{equation}

The \emph{dual} of a family $\{\beta_{i}\}_{i=1}^{n}$ of $n$ elements in $\MR^{n}$ is by definition another family of $n$ elements in $\MR^{n}$ $\{\alpha_i\}_{i=1}^{n}$ such that 
$$\langle \alpha_{i},\beta_{j}\rangle=\delta_{ij}\1\quad \text{for all $i,j$}.$$
If we write $\beta_i=(b_i+\sqrt{-1}c_i,v_i)$ and $\{b_i\}_{i=1}^n$ and $\{v_i\}_{i=1}^n$ are bases of $\R^n$ and $\Z^n$, respectively then $\{\beta_i\}_{i=1}^n$ has a dual $\{\alpha_i\}_{i=1}^n$. In this case,
If $\{\alpha_i\}_{i=1}^n$ is dual to $\{\beta_i\}_{i=1}^n$, then the composition 
$$\big(\prod_{i=1}^n \la_{\beta_i}\big)\big(\bigoplus_{i=1}^n\chi^{\alpha_i}\big)\colon (\C^*)^n\to (\C^*)^n$$ 
is the identity,  in particular, both $\bigoplus_{i=1}^n \chi^{\alpha_i}$ and $\prod_{i=1}^n \la_{\beta_i}$ are automorphisms of $(\C^*)^n$.

\begin{defi}
An \emph{abstract simplicial complex} $\Si$ is a collection of non-empty finite subsets of a fixed set, such that for every set $I\in \Si$ and $J\subset I$, we have $J\in \Si$. Let $\Si$ be an abstract simplicial complex. If $\Si$ contains only finitely many elements, then it is called a finite abstract simplicial complex. The dimension of $I\in \Si$ is $|I|-1$. If $|I|\leq n$ for all $I\in \Si$ and there exists $I_{0}\in \Si$ such that $|I_{0}|=n$, then $\Si$ is called an abstract simplicial complex of dimension $n-1$. An augmented abstract finite simplicial complex $\Si$ of dimension $n-1$ is a set $\Si'\bigcup \{\emptyset\}$ where $\Si'$ is an abstract finite simplicial complex of dimension $n-1$.
For $i\in \N$, denote by $\Si^{(i)}$ the set of elements $I\in \Si$ such that $|I|=i$. In particular, we denote the vertex set by $\Si^{(1)}$.
\end{defi}

\begin{defi}
Let $\Si$ be an augmented abstract finite simplicial complex of dimension $n-1$ and let 
\[
\beta\colon \Si^{(1)} \to \MR^n=\CZN
\]
where $\Si^{(1)}$ denotes the vertex set of $\Si$.  We abbreviate an element $\{i\}\in\Si^{(1)}$ as $i$ and $\beta(\{i\})$ as $\beta_i$ and express $\beta_i=(b_i+\sqrt{-1}c_i,v_i)\in\CZN$. Denote the cone spanned by $b_{i},i\in I$ by $\angle b_{I}$. Then the pair $(\Si,\beta)$ is called a \emph{(simplicial) topological fan} of dimension $n$ if the following are satisfied.
\begin{enumerate}
\item  $b_i$'s for $i\in I$ are linearly independent whenever $I\in \Si$, and $\angle \b_I\cap \angle \b_J=\angle \b_{I\cap J}$ for any $I,J\in \Si$.
(In short, the collection of cones $\angle b_I$ for all $I\in \Si$ is an ordinary simplicial fan although $\b_i$'s are not necessarily in $\Z^n$.) 
\item Each $\v_i$ is primitive and $v_i$'s for $i\in I$ are linearly independent (over $\R$) whenever $I\in \Si$.  
\end{enumerate}     
A topological fan $\De$ of dimension $n$ is \emph{complete} if $\bigcup_{I\in \Si}\angle \b_I=\R^n$ and \emph{non-singular} if the $\v_i$'s for $i\in I$ form a part of a $\Z$-basis of $\Z^n$ whenever $I\in\Si$.  
\end{defi}

\noindent 
{\bf Construction.}
Let $\De=(\Si,\beta)$ be a non-singular topological fan of dimension $n$. We take the vertex set $\Si^{(1)}$ as $[m]=\{1,2,\dots,m\}$. By definition, for all $I\in \Si$ we must have $I\subset [m]$. For $I\subset [m]$, we set $I^{c}:=[m]\setminus I$ and \[U(I):=\{(z_1,\dots,z_m)\in \C^m\mid z_i\not=0 \ \text{for $\forall i\notin I$}\}=\C^{I}\times (\CS)^{I^{c}}.\]
Note that $U(I)\cap U(J)=U(I\cap J)$ for any $I,J\in [m]$ and $U(I)\subset U(J)$ if and only if $I\subset J$.  We define 
\[
U(\Sigma):=\bigcup_{I\in \Si}U(I).
\]
Let 
\[
\la:(\C^*)^m\to (\C^*)^n
\]
be the homomorphism defined by
\[
\la(h_1,\dots,h_m):=\prod_{k=1}^m\la_{\beta_k}(h_k).
\]

It was shown in Lemma 4.1 of \cite{IFM12} that
$\lambda$ is surjective and
\begin{equation*}\label{kerl}
\Ker\la=\{(h_1,\dots,h_m)\in (\C^*)^m
\mid \prod_{k=1}^mh_k^{\langle \alpha,\beta_k\rangle}=1\quad \text{for any $\alpha\in \MR^n$}\}.
\end{equation*}
 For any $I\in \Si^{(n)}$ the dual of the family $\{\beta_i\}_{i\in I}$ is denoted by $\{\alpha^I_i\}_{i\in I}$. We then have  
\begin{equation*} \label{aibk}
\Ker\la=\{(h_1,\dots,h_m)\in (\C^*)^m \mid h_i\prod_{k\notin I}h_k^{\langle \alpha^I_i,\beta_k\rangle}=1\quad \text{for any $i\in I$}\}.
\end{equation*}

Denote $U(\Si)/\Ker \la $ by $X(\De)$, $U(I)/\Ker\la$ by $U_{I}$, and the representation space of $\bigoplus_{i\in I}\chi^{\alpha^{I}_{i}}$ by $\bigoplus_{i\in I}V(\chi^{\alpha^{I}_{i}})$. 
In Lemmas 5.1 and 5.2 of \cite{IFM12}, it is shown that
   $X(\De)$ is a smooth real manifold of dimension $2n$ with a smooth $\T=(\CS)^n\simeq (\CS)^m/\Ker \la$-action, whose equivariant local charts are the maps \[\varphi_{I}:U_{I}\to \bigoplus_{i\in I}V(\chi^{\alpha^{I}_{i}}) \qquad
[z_1,\dots,z_m]\mapsto (\prod_{k=1}^{m}z_{k}^{\langle \alpha^{I}_{i},\beta_{k}\rangle})_{i\in I} \] for all $I\in \Si^{(n)}$, and whose  transition functions are given by $$\varphi_{J}\varphi_{I}^{-1}:\varphi_{I}(U_{I}\cap U_{J})\to \varphi_{J}(U_I\cap U_{J}) 
\qquad (w_i)_{i\in I} \mapsto (\prod_{i\in I}w_{i}^{\langle \alpha^{J}_{j},\beta_{i}\rangle})_{j\in J}. $$ for all $I,J\in \Si^{(n)}$. In Section 3 of \cite{IFM12}, it is proven that that any topological toric manifold $X$, up to equivariant diffeomorphism, is of the form  $X(\De)$ for some complete nonsingular topological fan $\De$.

% \begin{defi}
% A closed smooth manifold of dimension $2n$ with an effective smooth action of $(\CS)^{n}$ having an open dense orbit is a topological toric manifold if it is covered by finitely many invariant open subsets each equivariantly diffeomorphic to a direct sum of complex one-dimensional smooth representation spaces of $(\CS)^{n}$. 
% \end{defi}
\bigskip
\section{Klyachko vector bundles}
Let $\De=(\Si,\beta)$ be a non-singular topological fan of dimension $n$.
For each $I\in \Si$, we call $$\gamma_{I}=[0^{I}\times 1^{I^c}]\in U(I)/\Ker \la$$ the \emph{distinguished point corresponding to $I$}. We denote the $\T$-orbit  of $\gamma_I$ by $$\mathcal{O}(I):=\T \cdot \gamma_{I}=(0^{I}\times (\C^{*})^{I^c})/\Ker\la,$$
and its closure in $X(\De)$ by 
$$V(I):=\overline{\MO(I)}=(0^{I}\times \C^{I^c})/\Ker\la.$$\\
The following theorem is a direct generalization of Theorem 3.2.6 of \cite{CLS}.
\begin{theo}[{\cite[Theorem 2.1]{Cui25}}] \label{orbit-cone}
$ $\\
    \begin{enumerate}
        \item There is a bijective correspondence 
            \begin{align*}
            \Si &\longleftrightarrow \{\T\text{-orbits in } X(\De)\}\qquad 
            I \mapsto \MO(I)
            \end{align*}
        \item $\forall\; J\in \Si^{(j)}\quad  \dim_{\C}\MO(J)=n-j$.
        \item $\displaystyle U_{I}=\bigcup_{J\leq I}\MO(J).$
        \item $J\leq I \quad \Leftrightarrow\quad \MO(I)\subset V(J)$, and in this case  $\displaystyle V(J)=\bigcup_{J\leq I}\MO(I).$
    \end{enumerate}
\end{theo}

In \cite{Cui25}, the first author classified topological/smooth/holomorphic Klyachko vector bundles over a topological toric manifold. 
\begin{defi}
A \emph{topological Klyachko vector bundle} of rank $r$ over a $2n$-dimensional topological toric manifold $X$ associated to a nonsingular topological fan $\Delta=(\Sigma,\beta)$ is a topological $\T$-equivariant complex vector bundle $\ME\to X$, whose restriction $\left. \ME \right|_{U_{I}}$ is equivariantly homeomorphic to $U_{I}\times \C^{r}$ for all $I\in \Si$, and whose action map $\T\times \ME\to \ME$ is continuous.
Denote by $\TKVB_{X}$ the category, whose objects are continuous Klyachko vector bundles over $X$, and whose morphisms are $\T$-equivariant continuous morphisms of complex vector bundles.
A \emph{smooth Klyachko vector bundle} is defined similarly by replacing continuous with smooth.
\end{defi}
\begin{defi} \label{geq}
We can define a partial order $\geq_{s}$ on $\MR$ as follows: For $\alpha=(b+c\sqrt{-1},v)\in \MR$, we say $\alpha \geq_{s} 0=(0+0\sqrt{-1},0)$ if $c=0, b\in \N, b\pm v\in 2\N$. For $\alpha,\beta\in \MR$, we say $\alpha\geq_{s}\beta$ if $\alpha-\beta\geq_{s} 0$.
For a topological toric manifold $X$, denote by $\SKVS_{X}$  the category, whose objects are complex vector spaces $E$ equipped with a poset $P$ of subspaces $E^{i}(\mu)$ of $E$ indexed by $\mu\in \MR$, $i\in \Si^{(1)}$ satisfying the following compatibility condition:\\
    For any $I\in \Si$ and $\beta^\perp_I:=\{\alpha\in \MR^m \mid \langle \alpha ,\beta_i\rangle=0 \;\forall i\in I\}$, there exists a $\MR^{n}/\bp$-grading   \[E=\bigoplus_{[\chi]\in \MR^{n}/\bp}E_{[\chi]} \quad \text{ for which } \quad
    E^{i}(\mu)=\sum_{\langle \chi,\beta_{i} \rangle \geq_{c}\mu}E_{[\chi]} \; \forall i\in I.\]
    A morphism $f:(E,\{E^{i}(\mu)\})\to (F,\{F^{i}(\mu)\})$ in $\TKVS_{X}$ is a $\C$-linear transformation $f:E\to F$ such that $f(E^{i}(\mu))\subset F^{i}(\mu)$ for all $i\in \Si^{(1)}$ and $\mu\in \MR$. %$f$ is an isomorphism if $E\to F$ is an isomorphism of complex vector spaces. $f$ is a monomorphism if $E\to F$ is a monomorphism of complex vector spaces. $f$ is an epimorphism if $E\to F$ is an epimorphism of complex vector spaces and $f(E^{i}(\mu))=F^{i}(\mu)$ for all $i\in \Si^{(1)}$ and $\mu\in \MR$.
\end{defi}

\begin{theo} [{\cite[Theorem 3.6]{Cui25}}]
    The category $\TKVB_{X}$ of topological Klyachko vector bundles over the topological toric manifold $X$ is equivalent to the category $\TKVS_{X}$. 
\end{theo}

\cite[Theorem 4.5]{Cui25} proves a parallel classification for the smooth Klyachko vector bundles.

%\begin{theo} [cf.~Theorem 3.6 of \cite{Cui25}]
 %   The category $\TKVB_{X}$ of topological Klyachko vector bundles over the topological toric manifold $X=X(\De)$ is equivalent to the category $\TKVS_{X}$.
%\end{theo}

\section{Equivariant triviality}
In this section, we prove the main result of the paper. For the benefit of the reader, we first treat the case of line bundles, which already involves most of the ideas. We then modify the proof in the line bundle case to obtain our result for all ranks. 

\noindent \textbf{Notation.} We will use $\cdot$ to denote the action of the torus $\T=(\C^{*})^n$ on itself, on a given $2n$-dimensional topological toric manifold $X$, and also on the fibers of an equivariant vector bundle $\ME\to X$.  
\bigskip
\begin{theo} \label{main}
    Let $X$ be the $2n$-dimensional topological toric manifold associated with a topological fan $\De$ and $\pi: \ME\to X$ a topological $\T$-equivariant complex line bundle, then $\ME$ is a topological Klyachko line bundle.
\end{theo}
\begin{proof}
    Given $I\in \Si^{(n)}$, we will give a $\T$-equivariant trivialization of  $\left.\ME \right|_{U_{I}}$. Let $x_{0}$ be the unique point fixed by $\T$ in $U_{I}$. Let $$ G:=(S^{1})^{n}< \T, \qquad \chi:\T\to \operatorname{GL}_{1}(\ME(x_{0}))=\CS$$ be respectively the maximal compact subgroup and the representation of $\T$ on the fiber $\ME(x_{0})$ of $\ME$ over $x_{0}$. 
    \bigskip
    
    {\bf Step 1:} We first prove that there is a $G$-equivariant trivialization of $\ME$ $$\varphi:\pi^{-1}(U)\to U\times \C$$ for some $G$-invariant open subset $U\subset U_{I}$ containing $x_{0}$.
    
    Let $\phi:\pi^{-1}(W)\to W\times \C$ be a trivialization of $\ME$ (not necessarily $G$-equivariant) over an open neighborhood $W$ of  $x_{0}$. Let $\langle\;,\,\rangle$ be an arbitrary Riemannian metric on $X$. Because $G$ is compact, it admits a Haar measure $\mu$, i.e. a  probability measure satisfying $$\int_{G}f(kh)~d\mu(h)=\int_{G}f(hk)~d\mu(h)=\int_{G}f(h)~d\mu(h)$$ for all $k, h\in G$ and $f\in L^{1}(G,\mu)$. Let $\langle \;,\,\rangle_{G}$ be the $G$-invariant Riemannian metric on $X$ defined by 
    \[ \langle v,w\rangle_{G}:=\int_{G}\langle d\rho_{h}(v),d\rho_{h}(w)\rangle~ d\mu(h) \qquad \forall\; p\in X,\; v,w\in T_{p}X,\]
    where $\rho_{h}:X\to X$ is the action of $h\in G$ on $X$.
    Let $d(-,-)$ be the distance function associated $\langle \;,\,\rangle_{G}$. Define $U$ to be the $G$-invariant open ball \[B(x_{0},\epsilon):=\{x\in W\mid d(x,x_{0})<\epsilon\}\subset W\] for some $\epsilon>0$. Define a $G$-action on $U\times \C$ via
    \[g\cdot (u,v)=(g\cdot u, \chi(g)v) \text{ for } (u,v)\in U\times \C, \] where the action on the first factor is the restriction of the $\T$-action on $X$.
    Let $\phi_{u}:\ME(u)\to \C$ be the linear isomorphism induced from $\phi$ for each $u\in U$. Define 
    \begin{align*}
        \varphi: \pi^{-1}(U)&\longrightarrow U\times \C\\
        e&\longmapsto (u,\OPHI_{u}(e)),
    \end{align*}
    where \[\OPHI_{u}(e)= \int_{G}\chi(g^{-1})\phi_{g\cdot u}(g\cdot e)~d\mu(g) \text{ with } u=\pi(e).\]
    $\varphi$ is clearly an isomorphism of bundles. We claim  that $\varphi$ is $G$-equivariant. To see this, we need to show that
    \[\varphi(h\cdot e)=(h\cdot u, \OPHI_{h\cdot u}(h\cdot e))=(h\cdot u,\chi(h)\OPHI_{u}(e))=h\cdot \varphi(e) \qquad \forall\; h\in G,\; e\in \pi^{-1}(U). \]
    On the one hand, 
    \begin{align*}
      \OPHI_{h\cdot u}(h\cdot e)&= \int_{G}\chi(g^{-1})\phi_{gh\cdot u}(gh\cdot e)~d\mu(g)  \\
      &=\int_{G}\chi((kh^{-1})^{-1})\phi_{k\cdot u}(k\cdot e)~d\mu(kh^{-1})\qquad (k:=gh)\\
      &=\int_{G}\chi(hk^{-1})\phi_{k\cdot u}(k\cdot e)~d\mu(k) \qquad \text{(by the invariance of $\mu$)}.\\
    \end{align*}
    On the other hand, 
    \begin{align*}
        \chi(h)\OPHI_{u}(e)=\chi(h)\int_{G}\chi(k^{-1})\phi_{k\cdot u}(k\cdot e)~d\mu(k).
    \end{align*}
    Comparing the last two equalities, we prove the claim.
    
    \bigskip 
    
    {\bf Step 2:} Define the section $s:U\to \ME |_U$ by
    \[s(u):=\varphi^{-1}((u,1)) \qquad \forall \; u\in U. \]
    We claim that $s$ is a local $G$-eigenframe, i.e. 
    \[ (g\cdot s)(u)=\chi(g)s(u) \qquad \forall  \; g\in G, \;u\in U. \]
    To see this, note that since $\varphi$ is $G$-equivariant, we can write  
    \[\varphi(g\cdot (s(u)))= g\cdot \varphi(s(u))= g\cdot (u,1)= (g\cdot u, \chi(g)). \]
   Therefore, 
   \[g\cdot (s(u))= \varphi^{-1}(g\cdot u, \chi(g))=\chi(g)\varphi^{-1}(g\cdot u,1)=\chi(g)s(g\cdot u), \]
   and hence 
   \[(g\cdot s)(u):=g\cdot (s(g^{-1}\cdot u))=\chi(g)(s(gg^{-1}\cdot u))=\chi(g)s(u) \] proving the claim.

    \bigskip
    
    {\bf Step 3:} Decompose $t\in \T=(\CS)^{n}$ as $t=r\cdot e^{i\theta}$, where \[r:=(r_{1},\dots,r_{n})\in \RP^{n},\qquad e^{i\theta}:=(e^{i\theta_{1}},\dots,e^{i\theta_{n}})\in G.\] Our goal is to extend the local $G$-eigenframe $s$ from Step 2 to a $\T$-eigenframe $$s:U_{I}\to \ME|_{U_I}.$$ For this, we need to find an extension of $s$ to $U_I$, such that
    \[(t\cdot s)(x):=t\cdot (s(t^{-1}\cdot x))=\chi^{\alpha}(t) s(x) \qquad \forall\; t\in \T, \; x\in U_{I}. \]
    for the continuous character $\chi^{\alpha}:\T\to \CS$ corresponding to some $\alpha \in \MR^n$, or equivalently
    \[s(t\cdot x)=\chi^{\alpha}(t^{-1}) t\cdot (s(x)) \qquad \forall\; t\in \T,\; x\in U_{I} .\]
    In the following, we will show that 
    \begin{align} \label{equi}
        s(t\cdot x)=\chi^{\alpha}(t^{-1})t\cdot (s(x)) \qquad \forall \; t\in \T, \;x\in U \text{ such that } t\cdot x\in U .
    \end{align}
    Once this is proven, we can define $s:U_{I}\to \ME$ by 
    \[s(t\cdot x):=\chi^{\alpha}(t^{-1}) t\cdot (s(x)) \qquad \forall \; t\in \T,\; x\in U \]
    whether $t\cdot x\in U$ or not. Because $\T\cdot U=U_{I}$, this definition indeed extends the local $G$-eigenframe $s:U\to \ME|_U$ to the $\T$-eigenframe $s:U_{I}\to \ME$, as desired. This $\T$-eigenframe will then give a $\T$-equivariant trivialization of $\ME|_{U_I}$ completing the proof of the theorem. 

    We now return to proving \eqref{equi}. The local section $s:U \to \ME|_U$ from Step 2 satisfies
    \begin{align} \label{equi2}
         s(e^{i\theta}\cdot x)=\chi(e^{-i\theta}) e^{i\theta}\cdot (s(x)) \qquad \forall \; e^{i\theta}\in G, \; x\in U.
    \end{align}
    
    Define $f:S:=\{(\rho,x)\in \RP^{n}\times U|x\in \rho\cdot  U \}\to \CS$ to be the continuous function satisfying 
    \[(r\cdot s)(x):=r\cdot (s(r^{-1}\cdot x))=f(r,x)s(x) \qquad \forall \; r\in\RP^{n},\; x\in U\bigcap (r\cdot U), \]
    or equivalently,
    \begin{align}\label{equi3}
        s(r\cdot x)=f(r^{-1},x) r\cdot (s(x)) \qquad \forall \; r\in\RP^{n}, \;x\in U\bigcap (r^{-1}\cdot U).
    \end{align}
    Given $r_{1},r_{2}\in \RP^{n}, x\in (r_{1}\cdot U)\bigcap (r_{2}\cdot U)\bigcap (r_{1}\cdot r_{2}\cdot U)$, we have 
    \begin{align*}
       f(r_{1}\cdot r_{2},x)s(x)&=((r_{1}\cdot r_{2})\cdot s)(x)\\
       &=(r_{1}\cdot (r_{2}\cdot s))(x)\\
       &=f(r_{1},x)(r_{2}\cdot s)(x)\\
       &=f(r_{1},x)f(r_{2},x)s(x)
    \end{align*}
    and since $s(x)\neq 0,$ we must have \begin{equation} \label{fmultip} f(r_{1}\cdot r_{2},x)=f(r_{1},x)f(r_{2},x).\end{equation}
    Now for any $t=r\cdot e^{i\theta}\in \T,\; x\in U\bigcap (r^{-1}\cdot U)$, we have
\begin{align}
        s(t\cdot x)&=s(e^{i\theta}\cdot r\cdot x)\notag \\ 
        &=\chi(e^{-i\theta})e^{i\theta}\cdot (s(r\cdot x))\notag \\ 
        &=\chi(e^{-i\theta})e^{i\theta}\cdot f(r^{-1},x)r\cdot (s(x)) \qquad (\text{by \eqref{equi3}})\notag \\ 
        &=\chi(e^{-i\theta})f(r^{-1},x) t\cdot(s(x))  \label{eq:final}
        \end{align}    
    
    and for $t=r\cdot e^{i\theta}\in \T,\;\rho\in \RP^{n},\; x\in U\bigcap( r^{-1}\cdot U)\bigcap (r^{-1}\cdot \rho \cdot  U)$,
    \begin{align*}
       f(\rho,t\cdot x)s(t\cdot x)&=(\rho\cdot s)(t\cdot x)\qquad (\text{by the definition of $f$})\\&=  \rho \cdot(s(\rho^{-1}\cdot t\cdot x))\\
        &=\rho \cdot (\chi(e^{-i\theta})f( r^{-1}\cdot \rho,x)\rho^{-1}\cdot t\cdot (s(x))) \qquad (\text{by \eqref{eq:final}})\\
        &=f(\rho,x) \chi(e^{-i\theta})f(r^{-1},x)t\cdot (s(x)) \qquad (\text{by \eqref{fmultip}}\\
        &=f(\rho,x)s(t\cdot x) \qquad (\text{by \eqref{eq:final}}).
    \end{align*}
     We have shown that \[f(\rho,t\cdot x)=f(\rho,x) \qquad \forall \; t\in \T,\; x\in U\bigcap (\rho\cdot  U) \text{ such that } t\cdot x\in U\bigcap (\rho\cdot  U).\] Given $x,y\in U\bigcap (\rho\cdot  U)$, there exists a sequence of points $t_{i}\in \T$ such that $$\lim_{i\to \infty}t_{i}\cdot x=y \quad \text{ and }\quad  t_{i}\cdot x\in U\bigcap (\rho\cdot  U).$$ By the continuity of $f$,
    \[f(\rho,y)=f(\rho,(\lim_{i\to \infty}t_{i})\cdot x)=\lim_{i\to \infty}f(\rho,t_{i}\cdot x)=f(\rho,x). \]
    Thus $f(\rho,x)$ is independent of $x$, so  we will denote it simply by $f(\rho)$. We can now rewrite ~(\ref{equi3}) as 
    \begin{align} \label{equi4}
         s(r\cdot x)=f(r^{-1})r\cdot (s(x)) \qquad \forall\;r\in\RP^{n},\; x\in U\bigcap (r^{-1}\cdot U).
    \end{align}
    Using ~(\ref{equi2}) and ~(\ref{equi4}), we see that ~(\ref{equi}) holds with the continuous $\T$-character
    \[\chi^{\alpha}(r\cdot e^{i\theta}):=f(r) \chi(e^{i\theta}) \qquad \forall \; r\in \RP^{n},\; \theta\in \R^{n} \] as claimed.
\end{proof}
\bigskip

\begin{theo} \label{main2}
    Let $X$ be the $2n$-dimensional topological toric manifold associated with a topological fan $\De$ and $\pi: \ME\to X$ a topological $\T$-equivariant complex vector bundle of rank $k$, then $\ME$ is a topological Klyachko vector bundle.
\end{theo}
\begin{proof} Given $I\in \Si^{(n)}$ and $x_{0}$ the unique point fixed by the action of $\T$ on $U_{I}$. Let $G=(S^{1})^{n}$ be the maximal compact subgroup of $\T$. 

\bigskip
{\bf Step 1:} Following Step 1 of the proof of Theorem~\ref{main}, we can find a $G$-invariant open ball $U=B(x_{0},\epsilon)$ over which there is a $G$-equivariant homeomorphism \[\varphi:\pi^{-1}(U)\to U\times \C^{k}.\]

\bigskip
{\bf Step 2:} Define the section $s_{i}:U\to \ME|_U$ by  \[s_{i}(u):=\varphi^{-1}((u,e_{i})) \qquad \forall\; u\in U,\] where $e_{i}$ is the $i$-th standard basis element of $\C^{k}$, and define $$s(u):=(s_{1}(u),\dots,s_{k}(u))^T.$$ As in Step 2 of the proof of Theorem~\ref{main}, we can check that $\{s_{1},\dots,s_{k}\}$ forms a local $G$-eigenframe, 
that is \[ (g\cdot s_{i})(u)=\chi_{i}(g)s_{i}(u) \qquad \forall \; g\in G,\; u\in U\] for some continuous characters $\chi_{i}:\T\to \CS$.
Equivalently, \begin{equation}\label {bse}s(e^{i\theta}\cdot x)= B(e^{-i\theta}) (e^{i\theta}\cdot (s(x))) \qquad \forall\; e^{i\theta}:=(e^{i\theta_1},\dots, e^{i\theta_n})\in G,\;  x\in U,\end{equation} 
where \[B:G\to \GL(n,\C)\qquad g\mapsto \diag(\chi_{1}(g),\cdots,\chi_{k}(g)).\]

\bigskip
{\bf Step 3:} Decompose $t\in \T=(\CS)^{n}$ as $t=r\cdot e^{i\theta}$, where \[r:=(r_{1},\dots,r_{n})\in \RP^{n},\qquad e^{i\theta}:=(e^{i\theta_{1}},\dots,e^{i\theta_{n}})\in G.\] Define $A\colon S:=\{(\rho,x)\in \RP^{n}\times U|x\in \rho\cdot U \}\to \GL(k,\C)$  to be the continuous map satisfying 
    \[(r\cdot s)(x):=r\cdot (s(r^{-1}\cdot x))=A(r,x)s(x)\qquad  \forall\; r\in\RP^{n},\;x\in U\bigcap (r\cdot U), \]
    or equivalently,
    \begin{align} \label{srx}
        s(r\cdot x)=A(r^{-1},x) r\cdot (s(x)) \qquad \forall\; r\in\RP^{n},\; x\in U\bigcap (r^{-1}\cdot U).
    \end{align}
    %Given $r_{1},r_{2}\in \RP^{n},\; x\in r_{1}\cdot U\bigcap r_{2}\cdot U\bigcap (r_{1}\cdot r_{2}\cdot U)$, we have 
    %\begin{align*}
     %  A(r_{1}\cdot r_{2},x)s(x)&=((r_{1}r_{2})\cdot s)(x)\\
      % &=(r_{1}\cdot (r_{2}\cdot s))(x)\\
      % &=A(r_{1},x)(r_{2}\cdot s)(x)\\
      % &=A(r_{1},x)A(r_{2},x)s(x)
    %\end{align*}
    %and $\{s_{1}(x),\cdots, s_{k}(x)\}$ is a $\C$-basis of $\ME(x)$, therefore 
    By the same argument leading to \eqref{fmultip}, we can show that   
    \begin{equation}\label{Amultip} A(r_{1},x)A(r_{2},x)=A(r_{1}\cdot r_{2},x)=A(r_{2}\cdot r_{1},x)=A(r_{2},x)A(r_{1},x)\end{equation} for any $r_{1},r_{2}\in \RP^{n},\; x\in (r_{1}\cdot U)\bigcap (r_{2}\cdot U)\bigcap (r_{1}\cdot r_{2}\cdot U)$ using that $\{s_{1}(x),\cdots, s_{k}(x)\}$ forms a $\C$-basis of the fiber $\ME(x)$. One can similarly see that \begin{equation} \label{AB=BA} 
A(r,x)B(e^{i\theta})=B(e^{i\theta})A(r,x)\qquad \forall \; e^{i\theta}\in G,\; r\in \RP^n, \; x\in U\bigcap (r\cdot U). \end{equation}
We also have that for any $t=r\cdot e^{i\theta}\in \T, \; x\in U\bigcap (r^{-1}\cdot U)$
\begin{align}
    s(t\cdot x)&=s(e^{i\theta}\cdot r\cdot x)\notag \\ 
    &=B(e^{-i\theta})(e^{i\theta}\cdot (s(r\cdot x))) \qquad (\text{by \eqref{bse}})\notag \\ 
    &= B(e^{-i\theta})(e^{i\theta}\cdot( A(r^{-1},x)r)\cdot (s(x))) \qquad (\text{by \eqref{srx}}) \notag \\ 
    &=B(e^{-i\theta})A(r^{-1},x) t\cdot(s(x)) , \label{eq:final2}
    \end{align}    
    and for any $t=r\cdot e^{i\theta} \in \T,\; \rho\in \RP^{n}, \; x\in U\bigcap (r^{-1}\cdot U)\bigcap (r^{-1}\cdot \rho \cdot U)$,
    \begin{align*}
        A(\rho,t\cdot x)s(t\cdot x)&=(\rho\cdot s)(t\cdot x) \qquad (\text{by the definition of $A$})\\&= \rho \cdot(s(\rho^{-1}\cdot t\cdot x))\\
        &=\rho \cdot (B(e^{-i\theta})A( r^{-1}\cdot \rho,x)\rho^{-1}\cdot t\cdot (s(x))) \qquad (\text{by \eqref{eq:final2}})\\
        &=A(\rho,x) B(e^{-i\theta})A(r^{-1},x)t\cdot (s(x))\qquad (\text{by \eqref{Amultip}, \eqref{AB=BA}})\\
        &=A(\rho,x)s(t\cdot x) \qquad (\text{by \eqref{eq:final2}}).
    \end{align*}
    Using that $\{s_{1}(t\cdot x),\cdots, s_{k}(t\cdot x)\}$ forms a $\C$-basis of the fiber $\ME(t \cdot x)$, we have shown that 
    \[A(\rho,x)=A(\rho,t\cdot x) \qquad \forall \; t\in \T, \; x\in U\bigcap (\rho \cdot U) \text{ such that } t\cdot x\in U\bigcap (\rho\cdot U). \]
    By the same argument as in Step 3 of the proof of Theorem \ref{main}, the identity above leads us to \[A(\rho,x)=A(\rho,y) \qquad \forall \; x,y\in U\bigcap (\rho\cdot U) \] that is $A(\rho, x)$ is independent of $x$, so we 
    will denote it simply by $A(\rho)$.
    
    \bigskip
{\bf Step 4:} 
Let $P\in \GL(k,\C)$ be chosen, so that  the coordinates of $$\Tilde{s}(x_{0}):=P \, s(x_{0})$$ form a $\T$-eigenbasis of the fiber $\ME(x_{0})$ (note that the coordinates of $s(x)$ form a $G$-eigenbasis of the fiber $\ME(x_0)$).  
Define 
\[\Tilde{s}(x):=P\, s(x) \qquad \forall \; x\in U,\] 
\[\Tilde{A}(r):=P\, A(r)\,P^{-1},\quad  \Tilde{B}(g):=P\, B(g)\, P^{-1} \qquad \forall \; r\in \RP^{n}, \; g\in G. \]
By the choice of $P$, we know $\Tilde{A}, \Tilde{B}$ must be diagonal matrices. 
Moreover, by \eqref{eq:final2} \[ (t\cdot \Tilde{s})(x)=\Tilde{A}(r)\, \Tilde{B}(e^{i\theta}) \,\Tilde{s}(x) \qquad \forall \; t=r\cdot e^{i\theta}\in \T,\; x\in U\bigcap (r\cdot U),  \]
or equivalently, 
\[ \Tilde{s}(t\cdot x)=\Tilde{A}(r^{-1})\, \Tilde{B}(e^{-i\theta})\, t\cdot (\Tilde{s}(x)) \qquad \forall \;t=r\cdot e^{i\theta}\in \T, x\in U\bigcap (r^{-1}\cdot U). \]
Extend $\Tilde{s}$ to the section $\Tilde{s}:U_{I}\to \ME|_{U_I}$ by defining
\[ \Tilde{s}(t\cdot x):=\Tilde{A}(r^{-1})\, \Tilde{B}(e^{-i\theta})\, t\cdot (\Tilde{s}(x)) \qquad \forall \; t=r\cdot e^{i\theta}\in \T, \; x\in U. \]
Let $\Tilde{\chi_{i}}(t)$ be the $(i,i)$-entry of the matrix $\Tilde{A}(r)\, \Tilde{B}(e^{i\theta})$ for $t=r\cdot e^{i\theta}$. Then the $i$-th coordinate $\Tilde{s}_{i}$ of $\Tilde{s}$ is a $\T$-eigensection over $U_{I}$ with the eigenfunction being the continuous character $\Tilde{\chi_{i}}$. Since $\{s_{1},\dots, s_{k}\}$ forms a $G$-eigenframe and $\Tilde{s}=P\, s$, $\{\Tilde{s}_{1},\dots,\Tilde{s}_{k}\}$ must be a $\T$-eigenframe, and hence $\ME |_{U_{I}}$ is $\T$-equivariantly homeomorphic to $U_{I}\times \C^{k}$ showing that $\ME$ is a topological Klyachko vector bundle.
\end{proof}
%In the proof of Theorem \ref{main2}, we did not use the smoothness of the equivariant vector bundle $\ME$ anywhere: In step 1, the existence of the open ball $U$ and a $G$-equivariant trivialization of $\ME$ over $U$ relies on the smoothness of $X$ and compactness of $G$. In step 2 and step 3, the maps $A,B$ are only required to be continuous and the independence of $A$ from $x$ derives from continuity of $A$. In step 4, smoothness of the continuous character $\chi^{\alpha}$ follows from the special property of $\T$. Therefore, we have the following theorem:

\begin{theo} \label{main3}
   Let $X$ be the $2n$-dimensional topological toric manifold associated with a topological fan $\De$ and $\pi: \ME\to X$ a smooth $\T$-equivariant complex vector bundle, then $\ME$ is a smooth Klyachko vector bundle. \qed
\end{theo}

\begin{proof}
The proof is almost the same as that of Theorem~\ref{main2}. We keep the notations and point out the slight modifications below. Given $I\in \Si^{(n)}$, in step 1 we pick a $G$-invariant open ball $U=B(x_{0},\epsilon)$ over which there is a $G$-equivariant isomorphism \[\varphi:\pi^{-1}(U)\to U\times \C^{k}.\] Step 2 and step 3 are the same except that the sections $s_{i}$ are smooth. In step 4, we construct a continuous $\T$-eigenframe $\{\Tilde{s}_{1},\dots,\Tilde{s}_{k}\}$ over $U_{I}$ where the eigenfunction of $\Tilde{s}_{i}$ is a continuous character $\Tilde{\chi}_{i}$ as before. By Lemma 1.1 of~\cite{Cui25}, $\Tilde{\chi}_{i}$ is smooth for all $i$, hence $\{\Tilde{s}_{1},\dots,\Tilde{s}_{k}\}$ is a smooth $\T$-eigenframe and so gives a smooth $\T$-equivariant trivialization of $\ME |_{U_{I}}$. 
\end{proof}

\end{document}